\documentclass[11pt,final]{article}
\usepackage{amsmath,amsthm,amsfonts,amssymb,graphicx,enumerate,psfrag}
\usepackage{mathrsfs}
\usepackage{fullpage}

\usepackage[colorlinks,citecolor=blue,urlcolor=blue]{hyperref}
\usepackage[utf8]{inputenc} 
\newtheorem{lemma}{Lemma}
\newtheorem{theorem}{Theorem}
\newtheorem{example}{Example}
\newtheorem{corollary}{Corollary}
\newtheorem{remark}{Remark}
\newtheorem{question}{Question}
\newcommand{\dN}{\mathbb {N}}
\newcommand{\dZ}{\mathbb {Z}}
\newcommand{\dR}{\mathbb {R}}
\newcommand{\lip}{\textsc{lip}}
\newcommand{\trel}{{\rm t}_{\textsc{rel}}}
\newcommand{\tmix}{{\rm t}_{\textsc{mix}}}
\newcommand{\dist}{\mathrm{dist}}
\newcommand{\EE}{{\mathbb{E}}}
\newcommand{\PP}{{\mathbb{P}}}
\newcommand{\cP}{\mathcal{P}}
\newcommand{\cF}{\mathcal {F}}
\newcommand{\dtv}{d_{\textsc{tv}}}
\newcommand{\tv}{\textsc{tv}}
\newcommand{\pmin}{P_{\mathrm{min}}}
\newcommand{\diam}{\mathrm{diam}}
\title{Mixing time and expansion of non-negatively curved Markov chains}
\author{Florentin Münch\footnote{Max Planck Institute for Mathematics in the Sciences, Leipzig, muench@mis.mpg.de}\ \ \& Justin Salez\footnote{Université Paris-Dauphine \& PSL, CEREMADE, Paris, justin.salez@dauphine.psl.eu}}
\begin{document}
\maketitle
\begin{abstract}
We establish  three remarkable consequences of non-negative curvature for sparse Markov chains. First, their conductance decreases logarithmically with  the number of states. Second, their displacement is at least diffusive until the mixing time. Third, they never exhibit the  cutoff phenomenon. The first result provides a nearly sharp quantitative answer to a classical question of Ollivier, Milman \& Naor. The second settles a conjecture of Lee and Peres for graphs with non-negative curvature. The third offers a striking counterpoint to the recently established cutoff for non-negatively curved chains with uniform expansion.
\end{abstract}
\section{Introduction}

In Riemaniann geometry, a lower bound  on the Ricci curvature classically  implies an array of powerful estimates for the underlying manifold, including diameter bounds, volume growth, comparison principles, splitting theorems, spectral estimates, and concentration inequalities  \cite{jost2017a}. Over the past decade, those remarkable implications have motivated the development of non-smooth analogues of curvature that can be applied to discrete geometries 
\cite{schmuckenschlager1998curvature,lin2010ricci,
forman2003bochner,erbar2012ricci,ollivier2009ricci,
jost2021characterizations,najman2017modern, munch2017ollivier}.
In particular, Ollivier   \cite{ollivier2009ricci} proposed a transportation-based definition  that makes sense on arbitrary metric spaces, hence in particular on graphs and Markov chains. 
Informally, a metric space has non-negative Ollivier-Ricci curvature if balls are at least as close to each other as their centers are.  The simplest example of a finite non-negatively curved graph is  a cycle. It is classical that this graph has poor expansion,  that the random walk on it exhibits a diffusive behavior, and that its mixing time is of the same order as the inverse  spectral gap. The aim of the present paper is to show that those three properties are in fact shared by all sparse Markov chains with non-negative curvature. Before we state our results in full generality,  let us describe their content in the simple but important special case of random walk on graphs.

\paragraph{Non-negatively curved graphs.}Let  $G=(V,E)$ be a finite simple graph, and let $P$ denote the random-walk transition matrix of $G$. Thus,  $P$ acts on any function $f\colon V\to\dR$ as follows: 
\begin{eqnarray*}
(Pf)(x) & = & \frac{1}{\deg(x)}\sum_{y\sim x}f(y),
\end{eqnarray*}
where the notation  $y\sim x$ indicates that $\{x,y\}\in E$.   Following Ollivier \cite{ollivier2009ricci,ollivier2010survey}, we say that $G$ has \emph{non-negative curvature} if $P$ contracts the Lipschitz norm $\|f\|_{\lip}:=\max_{y\sim x }{|f(y)-f(x)|}$, i.e. 
\begin{eqnarray*}
\|Pf\|_{\lip} & \le & \|f\|_{\lip}.
\end{eqnarray*}
This fundamental property is satisfied by many natural families of graphs, including all Abelian Cayley graphs and, more generally, all Cayley graphs whose generating set is conjugacy-invariant. 
Additional details,  including a more effective formulation in terms of couplings, will be provided in the next section  when we discuss curvature for general Markov chains.

\paragraph{Expansion.} Our first result concerns the expansion of non-negatively curved graphs.  Write $\partial A$ for the edge-boundary of a set $A\subseteq V$, and $\deg(A)$  for the sum of the degrees of all vertices in $A$. With this notation, the \emph{conductance} (also known as Cheeger constant, or bottleneck ratio) of $G$ is 
\begin{eqnarray*}
\Phi  & := & \min\left\{\frac{|\partial A|}{\deg(A)}\colon A\subseteq V,0<\deg(A)\le |E|\right\}.
\end{eqnarray*}
Sequences of bounded-degree graphs whose size diverges but whose conductance remains bounded away from zero are famously known as \emph{expanders}. Whether such graphs can have non-negative curvature is a natural question, first raised by Milman and Naor and then popularized in a survey by Ollivier \cite[Problem T]{ollivier2010survey}. The problem has remained open until very recently, when a negative answer was given by the second author  \cite{salez2021sparse}. Specifically, the latter used the notion of entropy for graph limits to prove that non-negative curvature and expansion are incompatible “at infinity”, and the conclusion was then  transferred to  finite graphs using a compactness argument. A clear drawback of this approach is its non-quantitative nature. In particular, the second author asked for a direct, quantitative relation between volume, degree and expansion on non-negatively curved graphs. This is precisely the content of our first main result. 
\begin{theorem}[Poor expansion]\label{th:nonexp}If $G$ has non-negative curvature, then 
\begin{eqnarray*}
\Phi & \le & c\sqrt{\frac{d\log d}{\log n}},
\end{eqnarray*}
where $n$ denotes the number of vertices of $G$, $d$ the maximum degree, and $c$ a universal constant.
\end{theorem}
In other words, large graphs can not simultaneously enjoy non-negative curvature and uniform expansion unless their maximum degree grows at least like $\log n/\log\log n$. We note that this is sharp up to the $\log\log n$ correction. Indeed, a celebrated result of Alon and Roichman asserts that random  Cayley graphs with logarithmic degrees have uniform expansion with high probability \cite{MR1262979}, and non-negative curvature can be enforced by specializing this result to Abelian groups.

\paragraph{Mixing times.} Our second result is a complete determination of the order of magnitude of the mixing time of all vertex-transitive graphs with bounded degrees and non-negative curvature. Suppose that $G$ is vertex-transitive, with degree $d$ and volume $n$. Fix an arbitrary origin $x\in V$ (the choice is irrelevant, by transitivity), and consider the lazy simple random walk on $G$ started at $x$, i.e. the Markov chain $(X_t)_{t\ge 0}$ on $V$ with initial condition $X_0=x$ and transition matrix $(P+I)/2$. The \emph{mixing time} of $G$ is a fundamental graph-theoretical parameter, defined as follows \cite{MR3726904}:
\begin{eqnarray*}
\tmix & := & \min\left\{t\in\dN\colon \max_{A\subseteq V}\left|\PP(X_t\in A)-\frac{|A|}{n}\right|\le \frac 14\right\}.
\end{eqnarray*}
An important, closely related quantity is the so-called \emph{relaxation time} (or inverse spectral gap)
\begin{eqnarray*}
\trel & := & \frac{1}{1-\lambda_{2}},
\end{eqnarray*}
where $1=\lambda_1>\lambda_2\ge\ldots\ge\lambda_n$ denote the ordered eigenvalues of $P$. It is classical that $\tmix \ge \trel$,  but this inequality can be off by a factor as large as $\log n$  (this is the case, e.g., for expanders).  

\begin{theorem}[Mixing times]\label{th:graphmix}All vertex-transitive graphs with non-negative curvature satisfy
\begin{eqnarray*}
\tmix \ \asymp_d \ \trel \ \asymp_d \ \frac{1}{\Phi^2},
\end{eqnarray*} 
where the notation $a\asymp_d b$ means that the ratio $a/b$ is bounded from above and  below by positive constants that depend only on the degree $d$. 
\end{theorem}
This has the following remarkable consequence. For a sequence of graphs $(G_n)_{n\ge 1}$, the condition 
\begin{eqnarray*}
\frac{\tmix(G_n)}{\trel(G_n)} & \xrightarrow[n\to\infty]{} & +\infty
\end{eqnarray*}
is known as the \emph{product condition}. It is necessary (and conjecturally also sufficient, at least for vertex-transitive graphs) for the occurrence of the so-called \emph{cutoff phenomenon}, a celebrated but still mysterious phase transition in the approach to equilibrium of certain Markov chains \cite{MR1374011,MR3726904}. Thus, Theorem \ref{th:graphmix} implies that  vertex-transitive graphs with fixed degree and non-negative curvature never exhibit cutoff. This stands in stark contrast with  recent results due to the second author, showing that many non-negatively curved graphs with logarithmic degree do exhibit cutoff  \cite{salez2021cutoff}. Interestingly, the conclusion of Theorem \ref{th:graphmix} is known to hold for  fixed-degree Cayley graphs of \emph{moderate growth} \cite{MR1254308}. This geometric condition was later shown to be equivalent to the much simpler requirement that the diameter is algebraically large in the volume \cite{MR3439705} (see the recent paper \cite{MR4253426} for an extension to vertex-transitive graphs). This raises the following  question. A positive answer would be surprisingly strong, but we have not been able to produce any counter-example.
\begin{question}[Moderate growth?]Do all non-negatively curved graphs with degree at most $d$ satisfy
\begin{eqnarray*}
\diam(G) & \ge & \varepsilon_d n^{\varepsilon_d},
\end{eqnarray*}
where $n$ is the number of vertices, and $\varepsilon_d>0$  a  constant depending only on $d$?
\end{question}

Indeed, this question was answered affirmatively in case of a modified Barky Emery curvature dimension condition in \cite{bauer2015li}, and later by the first author in case of the weaker, unmodified Bakry Emery curvature dimension condition \cite{munch2019li}.

\paragraph{Diffusivity.}Finally, our last result concerns the speed of random walk on vertex-transitive graphs with non-negative curvature. Many infinite graphs such as the line $\dZ$ are known to exhibit a \emph{diffusive} behavior, in the sense that the typical graph distance between $X_t$ and $X_0$ grows like $\sqrt{t}$. On a finite graph, the distance to the starting point can of course no longer grow indefinitely with time, but one may still hope for a diffusive behavior on \emph{appropriate time-scales}. This vague statement was recently given a powerful rigorous content by Lee and Peres \cite{MR3127886}, who showed that the simple random walk  on any finite vertex-transitive graph satisfies the diffusive lower-bound
\begin{eqnarray*}
 \EE[\dist(X_0,X_t)] & \ge & c\sqrt{\frac{t}{d}}
\end{eqnarray*}  for all $t\in[d,\trel]$, where $c>0$ is a universal constant. The graph $\dZ_2^{d}\times \dZ_n$ shows that this lower-bound is sharp. However, the authors conjectured that the time-scale on which the diffusive behavior remains valid should actually be much longer, namely, of order $\tmix$ \cite[Conjecture 2.5]{MR3127886}. 
Our second result confirms this prediction in the case of non-negatively curved graphs.
\begin{theorem}[Diffusive lower-bound]\label{th:speed}If $G$ is vertex-transitive and non-negatively curved, then
\begin{eqnarray*}
 \EE[\dist(X_0,X_t)] & \ge & c\sqrt{\frac{t}{d}},
\end{eqnarray*} 
for $t\in[d,\tmix]$, where $c$ is a universal constant.
\end{theorem}

We emphasize that our estimates are not restricted to simple random walks on graphs. Analogous results will be stated for general Markov chains with non-negative curvature. In particular, neither reversibility, nor even the symmetry of the support of $P$ are actually required for a version of Theorem \ref{th:nonexp} to hold. 
Ollivier curvature with respect to a directed metric has been explored before in \cite{yamada2016ricci,ozawa2020geometric,ozawa2022heat,
eidi2020ollivierJournal}. However the specific consequences of non-negative curvature seem to be unexplored 
to the best of our knowledge. Our general results are exposed in Section \ref{sec:main} below, and are proved in Section \ref{sec:proofs}.

\section{Main results}
\label{sec:main}

In the remainder of the paper, we consider an arbitrary, irreducible stochastic matrix  $P$ on a finite state space $V$.  A natural measure of the ``distance'' from a state $x\in V$ to a state $y\in V$ is the minimum number of transitions needed for the chain to move from $x$ to $y$, namely
\begin{eqnarray*}
\dist(x,y) & := & \min\left\{k\in\dN\colon P^k(x,y)>0\right\}.
\end{eqnarray*}
This quantity is  not necessarily symmetric, but it clearly satisfies the two other axioms of a distance. We may then use optimal transport to extend this notion to probability measures as follows: write  $\cP(V)$ for the set of probability measures on $V$, and define $W\colon \cP(V)\times\cP(V)\to\dR_+$ by
\begin{eqnarray*}
W(\mu,\nu) & := & \inf_{X\sim\mu,Y\sim\nu}
\EE\left[\dist(X,Y)\right],
\end{eqnarray*}
where the infimum runs over all possible random pairs $(X,Y)$ whose marginals are $\mu$ and $\nu$. Again, this quantity is not necessarily symmetric, but it always satisfies the two other axioms of a distance. 
Due to Kantorovich duality \cite[Theorem~5.10 and Particular Case 5.4]{villani2009optimal}, we can write
\[
W(\mu,\nu) = \sup \left\{ \nu f - \mu f : \|f\|_{\lip} \leq 1 \right\}
\]
with 
\[
\|f\|_\lip := \sup_{y\sim x} f(y)-f(x). 
\]
Finally, we say that $P$ has \emph{non-negative curvature} if it is a contraction under $W$, i.e.
\begin{eqnarray}
\label{def:curvature}
\forall \mu,\nu\in\cP(V),\quad
W(\mu P,\nu P) & \le & W(\mu,\nu).
\end{eqnarray}
Due to Kantorovich duality and as in the introduction, non-negative is equivalent to
\[
\|Pf\|_\lip \leq \|f\|_\lip.
\]
For a continuous time version in the directed graph case, see \cite{ozawa2022heat}. 
Indeed, Ollivier curvature with non-symmetric distance has been studied before in \cite{yamada2016ricci,ozawa2020geometric,eidi2020ollivierJournal}.
By convexity, it is in fact sufficient to check property \eqref{def:curvature} on Dirac masses $\mu=\delta_x,\nu=\delta_y$, $x,y\in V$. Moreover, by the triangle inequality, we may further restrict our attention to the case where $y$ is a neighbor of $x$ (by which we mean that $\dist(x,y)=1$ and which we denote by $y\sim x$), i.e.
\begin{eqnarray}
\label{def:local}
\forall y\sim x, \quad
W\left(P(x,\cdot),P(y,\cdot)\right) & \le & 1.
\end{eqnarray}
This local condition is  easily verified in practice. For example, it  holds for  random walks on Abelian groups and, more generally,  random walks with a conjugacy-invariant support, as we now explain. 
\begin{example}[Random walks on groups] Suppose that $V$ is a group, and fix  $\mu\in\cP(V)$. By definition, the \emph{random walk} on $V$ with increment distribution $\mu$ is  the Markov chain whose transitions correspond to left-multiplication by a $\mu-$distributed element, i.e.
$
P(x,y)  :=  \mu(y x^{-1}).
$
This chain has non-negative curvature as soon as the set $S:=\{z\in V\colon\mu(z)>0\}$  is \emph{conjugacy-invariant}, i.e.
\begin{eqnarray}
\label{assume:conjugacy}
\forall z\in V,\quad z S z^{-1} & = &S.
\end{eqnarray}
Indeed, this assumption  implies that $\dist(zx,zy)=\dist(x,y)$ for all $x,y,z\in V$.  In particular, if $Z$ denotes a random variable with law $\mu$, then the ``obvious" coupling  of $P(x,\cdot)$ and $P(y,\cdot)$ given by $X:=Zx$ and $Y:=Zy$ verifies (\ref{def:local}). Note that the condition (\ref{assume:conjugacy}) trivially holds if the group is Abelian. An emblematic non-Abelian example is the transposition walk on the symmetric group \cite{MR3936154}. 
\end{example}

To avoid periodicity issues, we now assume that $P$ is lazy, i.e., $P(x,x)\ge 1/2$ for all $x\in V$. This is more than enough to guarantee that the chain mixes, in the sense that
 \begin{eqnarray*}
\forall x,y\in V,\quad P^t(x,y) & \xrightarrow[t\to\infty]{} & \pi(y),
 \end{eqnarray*}
where $\pi=\pi P$ denotes the unique invariant distribution. Quantifying the speed at which this convergence to equilibrium occurs is a fundamental question, with many applications \cite{MR3726904,MR2341319}. Formally, this amounts to estimating the so-called \emph{mixing time} of the chain:
\begin{eqnarray*}
\tmix & := & \min\left\{t\ge 0\colon \dtv(t)\le \frac{1}{4}\right\}, \quad\textrm{ where }\quad \dtv(t) := \max_{x\in V}\left\|P^t(x,\cdot)-\pi\right\|_{\tv}.
\end{eqnarray*}
Here $\|\mu-\nu\|_{\tv}$ denotes the \emph{total-variation distance} between $\mu,\nu\in\cP(V)$, defined as 
\begin{eqnarray*}
\left\|\mu-\nu\right\|_{\tv}  & = &  \max_{A\subseteq V}|\mu(A)-\nu(A)| \ = \ \frac{1}{2}\sum_{x\in V}|\mu(x)-\nu(x)| \ = \ \inf_{X\sim \mu,Y\sim\nu}\PP(X\ne Y),
\end{eqnarray*}
where the infimum in the last expression runs over all possible couplings $(X,Y)$ of $\mu$ and $\nu$. Thus, a natural way to estimate mixing times is to exhibit good couplings, and this is precisely where curvature enters the play. Indeed, an elementary but crucial reformulation of the non-negative curvature assumption (\ref{def:curvature}) is that the trajectories $(X_t)_{t\ge 0}$ and $(Y_t)_{t\ge 0}$ emanating from any two states $X_0=x$ and $Y_0=y$ can be coupled in such a way that their distance  $t\mapsto \dist(X_t,Y_t)$ forms a super-martingale. When combined with an appropriate diffusive estimate for super-martingales,  this  observation turns out to imply the following $O(1/\sqrt{t})$ decay for the total-variation distance between the laws of $X_t$ and $Y_t$.
\begin{theorem}[Total-variation decay]\label{th:tvdecay}Suppose that $P$ is lazy and non-negatively curved. Then,
\begin{eqnarray*}
\left\|P^t(x,\cdot)-P^t(y,\cdot)\right\|_{\tv} & \le & \dist(x,y)\sqrt{\frac{10\,}{{(t+1)\pmin }}},
\end{eqnarray*} 
for all $x,y\in V$ and all $t\ge 0$, where $\pmin$ denotes the smallest non-zero entry of $P$.
\end{theorem}
Variants of this result have appeared in a number of works, under various forms  \cite{MR2316551,lin2015equivalent,munch2019non,salez2021sparse}. However, all proofs use the fact that the increments of the process $t\mapsto \dist(X_t,Y_t)$ are uniformly bounded (by $2$), and this property may dramatically fail in our more general setup where the metric is directed. Nevertheless, the conclusion turns out to remain valid, and a proof is presented in Section \ref{sec:tvdecay}. The most ``obvious'' application of Theorem \ref{th:tvdecay} consists in  taking a maximum over all states $x,y\in V$ to obtain the following mixing-time estimate, which is new in our directed setup.
\begin{corollary}[Diameter bound]\label{co:diam}If $P$ is lazy and non-negatively curved, then
\begin{eqnarray*}
\tmix & \le & \frac{160\,(\diam)^2}{\pmin},
\end{eqnarray*}
where $\diam:=\max_{x,y}\dist(x,y)$ denotes the diameter of the state space.
\end{corollary}
While interesting in its own right, this estimate is actually not the key to the new results mentioned in the Introduction. Our main finding is that a  significantly finer estimate  can be deduced from Theorem \ref{th:tvdecay} provided we replace the \emph{worst-case} mixing time  by its \emph{average} version:
\begin{eqnarray*}
\tmix^{\sharp} & := & \min\left\{t \ge 0\colon \dtv^\sharp(t)\le \frac 14\right\},\quad\textrm{ where }\quad \dtv^\sharp(t):=\sum_{x\in V}\pi(x)\left\|P^t(x,\cdot)-\pi\right\|_{\tv}.
\end{eqnarray*}
\begin{remark}[Transitive chains]Obtaining a bound on  $\tmix^\sharp$ rather than $\tmix$ is not a huge drawback. For example, we have $\tmix^\sharp=\tmix$ for all random walks on groups and, more generally,   for all \emph{transitive} chains ($P$ is transitive if for each  $x,y \in V$, there is a bijection $f\colon V\to V$ which maps $x$ to $y$ and preserves the transition kernel, i.e., $P(f(u),f(v))=P(u,v)$ for all $u,v \in V$). 
\end{remark}
Throughout the paper, we let $X=(X_t)_{t\ge 0}$ denote a Markov chain with transition matrix $P$ starting from stationarity ($X_0\sim \pi$). Our main new estimate on $\tmix^\sharp$ will depend on  two key statistics of this chain. The first one is the \emph{mean displacement}  at time $t$: \begin{eqnarray*}
\EE[\dist(X_0,X_t)]& = & \sum_{x,y\in V}\pi(x)P^t(x,y)\dist(x,y).
\end{eqnarray*}
The second is the \emph{escape probability} in $t$ steps, i.e. the conductance of the $t$-th power of $P$:
\begin{eqnarray*}
\Phi(P^t) & = & \min\left\{\PP\left(X_t\notin A|X_0\in A\right)\colon A\subseteq V,0<\pi(A)\le \frac 12\right\}\\
& = &\min\left\{\frac{1}{\pi(A)}\sum_{x\in A}\sum_{y\in A^c}\pi(x)P^t(x,y)\colon A\subseteq V,0<\pi(A)\le \frac 12\right\}.
\end{eqnarray*}
\begin{theorem}[Main estimate]\label{th:main} If $P$ is lazy and non-negatively curved, then 
\begin{eqnarray}
\label{infimum}
\tmix^\sharp & \le & \frac{160}{\pmin}\,\inf_{t\ge 1}\left\{\frac{\EE[\dist(X_0,X_t)]}{\Phi(P^t)}\right\}^2,
\end{eqnarray}
where we recall that $\pmin$ denotes the smallest non-zero entry of $P$.
\end{theorem}

Theorem \ref{th:main}  has a number of notable consequences, which we now enumerate. The simplest one is an ``average'' version of Corollary \ref{co:diam}, obtained by  sending $t\to\infty$ in the infimum (\ref{infimum}): 
\begin{eqnarray}
\label{effectivediam}
\tmix^\sharp & \le & \frac{640\,(\diam^\sharp)^2}{\pmin},
\end{eqnarray}
where $
\diam^\sharp  :=  \sum_{x,y\in V}\pi(x)\pi(y)\dist(x,y)
$ denotes the \emph{effective diameter}. Note that the latter can be significantly smaller than the true diameter appearing in Corollary \ref{co:diam} (consider, e.g., the biased random walk on a  segment). A much more refined consequence of Theorem \ref{th:main} is obtained  by taking $t=1$ in the infimum (\ref{infimum}): writing $\Phi=\Phi(P)$, we readily obtain the following surprising bound.
\begin{corollary}[Conductance bound]\label{co:conductance}If $P$ is lazy and non-negatively curved, then
\begin{eqnarray*}
\tmix^\sharp & \le & \frac{40}{\pmin\Phi^2}.
\end{eqnarray*}
\end{corollary}
This  offers a considerable improvement over (\ref{effectivediam}) in situations where the effective diameter diverges while the conductance remains bounded away from $0$ (consider, e.g.,  random walk on a  random Abelian Cayley graph with logarithmic degree). More importantly, by virtue of an elementary combinatorial lower-bound on $\tmix^\sharp$ (see, e.g., \cite[Section 7.1.1]{MR3726904}), Corollary \ref{co:conductance} implies the quantitative non-existence of non-negatively curved expanders promised in Theorem \ref{th:nonexp}. For general chains, we will show that $\tmix^\sharp$ can be bounded below by $\diam^\sharp$, leading to the following result.
\begin{corollary}[Poor expansion]\label{co:expansion}If $P$ is non-negatively curved, then
\begin{eqnarray*}
\Phi & \le & \frac{19}{\sqrt{\pmin\diam^\sharp}}.
\end{eqnarray*}
\end{corollary}
Thus, non-negatively curved chains which are large ($\diam^\sharp\gg 1$) and  sparse ($\pmin$ bounded away from $0$) must have poor expansion ($\Phi\ll 1$). This constitutes a  precise quantitative answer to the Markov-chain generalization of the question of Milman, Naor and Ollivier \cite[Problem T]{ollivier2010survey}.  Note that there are examples of sparse chains with non-negative curvature and arbitrarily many states (consider, e.g., a biased random walk on a segment). However, the fact that their effective diameter is  bounded forces  their stationary measure to concentrate on a bounded number of states.

 Corollary \ref{co:conductance} is sharp in the important case  where $P$ is transitive, reversible and sparse. Indeed, we have the classical lower-bound $\tmix\ge\trel$, where $\trel:=(1-\lambda_2)^{-1}$ denotes the inverse spectral gap of $P$ (see, e.g., \cite{MR3726904}), and the first author established in \cite{munch2019non} the Buser inequality
\begin{eqnarray*}
\trel & \ge & \frac{\pmin}{12\Phi^2},
\end{eqnarray*}
for any non-negatively curved, reversible chain. When combined with Corollary \ref{co:conductance}, this yields the following result, of which Theorem \ref{th:graphmix} is clearly a special case.
\begin{corollary}[No cutoff for sparse  chains]\label{co:cutoff}Fix $p\in(0,1)$. Then, any lazy reversible transitive chain with non-negative curvature and $\pmin\ge p$ satisfies
\begin{eqnarray*}
\tmix \ \asymp_p \ \trel \ \asymp_p \ \frac{1}{\Phi^2},
\end{eqnarray*} 
where the notation $a\asymp_p b$ means that the ratio $a/b$ is bounded from above and  below by positive constants that depend only on $p$. In particular, no family of such chains can exhibit  cutoff.
\end{corollary}
An important observation here is that the transitivity of the chain is only used to ensure that $\tmix=\tmix^\sharp$. Consequently, Corollary \ref{co:cutoff} extends to any collection of chains which are ``spatially homogeneous'' in the mild sense that $\tmix\asymp\tmix^\sharp$. 

Finally, a last notable consequence of Theorem \ref{th:main} is that the expected displacement of the chain over short time-scales is already substantial. More precisely, assuming that $P$ is reversible, we have 
\begin{eqnarray*}
\Phi(P^t) & \ge & \frac{1-\lambda_2^t}{2}\ \ge \ \frac{1-e^{-\frac{t}{\trel}}}{2},
\end{eqnarray*}
and the right-hand side is at least $\frac{1-e^{-1}}{2}$ for all $t\ge \trel$, yielding the following estimate. 

\begin{corollary}[Fast escape]\label{co:escape}If $P$ is lazy, reversible and non-negatively curved, then 
\begin{eqnarray*}
\forall t\ge\trel,\quad \EE[\dist(X_0,X_t)] & \ge & \frac{\sqrt{\tmix^\sharp\pmin}}{41},
\end{eqnarray*}
where we recall that $\EE[\dist(X_0,X_t)]=\sum_{x,y}\pi(x)P^t(x,y)\dist(x,y)$.
\end{corollary}
For reversible transitive chains, Lee and Peres \cite{MR3127886} proved the diffusive lower-bound
\begin{eqnarray*}
 \EE[\dist(X_0,X_t)] & \ge & c\sqrt{t\pmin},
\end{eqnarray*}  for all $\pmin^{-1}\le t\le \trel$, where $c>0$ is a universal constant. They conjectured that this diffusive lower-bound should remain valid until the mixing time \cite[Conjecture 2.5]{MR3127886}.  Corollary \ref{co:escape} readily implies that this is true in the non-negatively curved case, and Theorem \ref{th:speed} follows as a special case. 

\section{Proofs}\label{sec:proofs}  
Section \ref{sec:conductance} below is devoted to the proof of our main result, namely the relation between conductance, displacement and mixing times (Theorem \ref{th:main}). The latter exploits the diffusive total-variation decay of non-negatively curved chains (Theorem~\ref{th:tvdecay}), which will be proven independently in Section \ref{sec:tvdecay}.
Once Theorem \ref{th:main} is established, all announced corollaries follow effortlessly, except for Corollary~\ref{co:expansion}: the latter requires a lower bound on the average mixing time in terms of the effective diameter, which we prove in Section \ref{sec:diameter}.

\subsection{Mixing time vs. conductance}
\label{sec:conductance}

In this section, we  prove Theorem \ref{th:main}. 
 We will make crucial use of Theorem \ref{th:tvdecay}, as well as the following $L^1$ version of Cheeger's inequality. An important remark is that the latter holds without any assumption on the transition matrix $P$, and will thus also apply to powers of $P$. 
\begin{lemma}[$L^1$ analogue of Cheeger's inequality]\label{lm:L1} If $f\colon V\to\dR$ satisfies $\pi f=0$, then,
\begin{eqnarray*}
 \sum_{x\in V}\pi(x)\left|f(x)\right| & \le & \frac{1}{\Phi(P)}\sum_{x,y\in V}\pi(x)P(x,y)\left|f(y)-f(x)\right|.
\end{eqnarray*}
\end{lemma}
\begin{proof} Upon replacing $f$ with $-f$, we may assume that $\pi(f\ge 0)\ge 1/2$. For any $t\ge 0$, we may take $A=\{f\ge t\}$ in the definition of $\Phi(P)$ to obtain
\begin{eqnarray*}
\Phi(P) \sum_{x\in V}\pi(x){\bf 1}_{(f(x)\ge t)} & \le & \sum_{x,y\in V}\pi(x)P(x,y){\bf 1}_{(f(y)\le t<f(x))}.
\end{eqnarray*}
Integrating over $t\in\dR_+$ and interchanging the sum and integral, we obtain
\begin{eqnarray*}
\Phi(P)\sum_{x\in V}\pi(x)f_+(x) & \le & \sum_{x,y\in V}\pi(x)P(x,y)\left(f_+(x)-f_+(y)\right)_+,
\end{eqnarray*}
where $a_+:=\max(0,a)$ denotes the positive part of $a$. Now, since $f$ is centered under $\pi$,  the left-hand side does not change if we replace $f_+(x)$ by $|f(x)|/2$. Similarly, since any gradient is centered under the measure $(x,y)\mapsto \pi(x)P(x,y)$, the right-hand side does not change if we replace $\left(f_+(x)-f_+(y)\right)_+$ by $\left|f_+(x)-f_+(y)\right|/2$. Finally, observe that $\left|f_+(x)-f_+(y)\right| \le|f(x)-f(y)|$.
\end{proof}
\begin{proof}[Proof of Theorem \ref{th:main}]Fix $s\ge 1$. Lemma \ref{lm:L1} applied to $P^s$ instead of $P$ gives
\begin{eqnarray*}
\sum_{x\in V}\pi(x)\left|f(x)\right|& \le & \frac{1}{\Phi(P^s)}\sum_{x,y\in V}\pi(x)P^s(x,y)\left|f(x)-f(y)\right|,
\end{eqnarray*}
for any centered observable $f\colon V\to\dR$. Now, fix $t\in\dN$ and $z\in V$, and let us apply this to the observable $f(x):=P^t(x,z)-\pi(z)$, which is centered because $\pi P^t=\pi$. We readily obtain
\begin{eqnarray*}
\sum_{x\in V}\pi(x)\left|P^t(x,z)-\pi(z)\right|& \le & \frac{1}{\Phi(P^s)}\sum_{x,y\in V}\pi(x)P^s(x,y)\left|P^t(x,z)-P^t(y,z)\right|.
\end{eqnarray*}
We may now sum over all $z\in V$ and use Theorem \ref{th:tvdecay} to get
\begin{eqnarray*}
\dtv^\sharp(t) & \le & \frac{1}{\Phi(P^s)}\sum_{x,y\in V}\pi(x)P^s(x,y)\left\|P^t(x,\cdot)-P^t(y,\cdot)\right\|_{\textsc{tv}}\\
& \le & \sqrt{\frac{10}{(t+1)\pmin}}\frac{\EE\left[\dist(X_0,X_s)\right]}{\Phi(P^s)}. 
\end{eqnarray*}
Finally, choosing $t$ so that the right-hand side is smaller than $1/4$ shows that
\begin{eqnarray*}
\tmix^\sharp & \le & \frac{160}{\pmin}\left(\frac{\EE\left[\dist(X_0,X_s)\right]}{\Phi(P^s)}\right)^2.
\end{eqnarray*}
The result follows by taking an infimum over all $s\ge 1$. 
\end{proof}

\subsection{Effective diameter vs. conductance}\label{sec:diameter}
Here we prove Corollary \ref{co:expansion}. We will use the following concentration inequality. 
\begin{lemma}[Concentration inequality]\label{lm:concentration}For any $f\colon V\to\dR$ and $a>0$,
\begin{eqnarray*}
\pi\left(f\ge \pi f+ a\right) & \le & \frac{1}{a\Phi}\min\left(\max_{y\sim x}\left(f(y)-f(x)\right)_+,\max_{y\sim x}\left(f(x)-f(y)\right)_+\right),
\end{eqnarray*}
and the same holds for $\pi(f\le \pi f- a)$.
\end{lemma}
We remark that if "$\sim$" is symmetric, then the minimum is the Lipschitz constant of $f$.

\begin{proof}
Upon replacing $f$ by $f-\pi f$ if necessary, we may assume that $f$ is centered under $\pi$, i.e. $\pi f=0$. Now, we use Markov's inequality and Lemma \ref{lm:L1} to write
\begin{eqnarray*}
\pi\left(f\ge a\right) & \le &  \frac{1}{a}\sum_{x\in V}\pi(x)f_+(x)\\
& =&  \frac{1}{2a}\sum_{x\in V}\pi(x)|f(x)|\\
& \le &\frac{1}{2a\Phi}\sum_{x,y\in V}\pi(x)P(x,y)\left|f(x)-f(y)\right|\\
& = & \frac{1}{a\Phi}\sum_{x,y\in V}\pi(x)P(x,y)\left(f(x)-f(y)\right)_+\\
& \le & \frac{1}{a\Phi}\max_{y\sim x}\left(f(x)-f(y)\right)_+.
\end{eqnarray*}
Note that, since the function $(x,y)\mapsto f(x)-f(y)$ is centered under the measure $(x,y)\mapsto \pi(x)P(x,y)$, the integral of its positive part equals that of its negative part. This establishes the first claim, and the second is obtained by replacing $f$ with $-f$. 
\end{proof}
We will use the above lemma to prove the following lower-bound on the mixing time.
\begin{lemma}[Diameter lower-bound]\label{lm:diam}For any lazy chain $P$, we have $\tmix^\sharp \ge \diam^\sharp-\frac{4}{\Phi}$.
\end{lemma}
\begin{proof}Let us first note that for a lazy chain, Lemma \ref{lm:concentration} holds with the better constant $2\Phi$ instead of $\Phi$ (just apply the lemma to the non-lazy chain $2P-I$). Now, fix $x\in V$ and $t\in\dN$, and write $B_x(t):=\{y\in V\colon \dist(x,y)\le t\}$. By definition,
\begin{eqnarray*}
\|P^t(x,\cdot)-\pi\|_\tv & = & \max_{A\subseteq V}|P^t(x,A)-\pi(A)|\\
& \ge & 1-\pi\left(B_x(t)\right)\\
& = & 1-\PP\left(\dist(x,Y)\le t\right),
\end{eqnarray*}
where $Y$ denotes a $\pi-$distributed random variable. Averaging over all  $x\in V$, we obtain
\begin{eqnarray*}
\dtv^\sharp(t) & \ge & 1-\PP\left(\dist(X,Y)\le t\right),
\end{eqnarray*}
where $X$ is $\pi-$distributed and independent of $Y$. Thus, the claim follows if we can show
\begin{eqnarray}
\label{toshow}
\PP\left(\dist(X,Y)\le \diam^\sharp-\frac{4}{\Phi}\right) & \le & \frac 12.
\end{eqnarray}
Now, for fixed $x\in V$, the function $f\colon y\mapsto \dist(x,y)$ satisfies $f(z)\le f(y)+1$ whenever $z\sim y$, by the triangle inequality. Thus, Lemma \ref{lm:concentration} ensures that
\begin{eqnarray*}
\PP\left(\dist(X,Y)\le \EE[\dist(X,Y)|X]-\frac{2}{\Phi}\right) & \le & \frac{1}{4}.
\end{eqnarray*}
On the other hand, by the triangle inequality again, the function $f\colon x\mapsto \EE[\dist(x,Y)]$  satisfies $f(x)\le f(y)+1$ whenever $y\sim x$, and $\pi f=\EE[\dist(X,Y)]=\diam^\sharp$. Thus, Lemma \ref{lm:concentration}  yields 
\begin{eqnarray*}
\PP\left(\EE[\dist(X,Y)|X]\le \diam^\sharp-\frac{2}{\Phi}\right) & \le & \frac{1}{4}. 
\end{eqnarray*}
Combining those two estimates readily yields (\ref{toshow}).
\end{proof}

\begin{proof}[Proof of Corollary \ref{co:expansion}]If $P$ is lazy and non-negatively curved, then Corollary \ref{co:conductance} and Lemma \ref{lm:diam} give
\begin{eqnarray*}
\diam^\sharp & \le & \frac{40}{\Phi^2\pmin}+\frac{4}{\Phi} \ \le \ \frac{41}{\Phi^2\pmin},
\end{eqnarray*}
because $\pmin,\Phi\le \frac 12$. If $P$ is not lazy, we apply the above result to  $(P+I)/2$. The latter is still non-negatively curved, but its conductance and minimal entry are half those of $P$, so we loose a factor of $8$ and obtain
\begin{eqnarray*}
\diam^\sharp & \le & \frac{328}{\Phi^2\pmin}.
\end{eqnarray*}
This readily implies the claim, because $\sqrt{328}<19$.
\end{proof}
\subsection{Diffusive total-variation decay}
\label{sec:tvdecay}
In this section, we prove Theorem \ref{th:tvdecay}. Fix two distinct states $x\ne y\in V$, and recall that
\begin{eqnarray*}
W\left(P(x,\cdot),P(y,\cdot)\right) & = & \inf_{\chi}\left\{\sum_{u,v\in V}\chi(u,v)\,\dist(u,v)\right\},
\end{eqnarray*}
where the infimum runs over all  probability distributions $\chi\in\cP(V^2)$ with marginals $P(x,\cdot)$ and $P(y,\cdot)$. Minimizers are called optimal couplings.   As in  \cite{MR2316551,munch2019non}, our first task consists in showing that they can  be chosen so as to assign a ``decent'' probability to the ``good'' set
\begin{eqnarray*}
\Gamma & := & \left\{(u,v)\in V^2\colon \dist(u,v)<\dist(x,y)\right\}.
\end{eqnarray*}
\begin{lemma}[Good optimal couplings] \label{lm:coupling}If $P$ is lazy and $x\ne y$, then there is an optimal coupling $\chi$ of $P(x,\cdot),P(y,\cdot)$ such that $\chi\left(\Gamma\right) \ge \pmin.$
\end{lemma}
\begin{proof}  
By compactness, we can find an optimal coupling $\chi$ which, among all optimal couplings, maximizes   $\chi(\Gamma)$. Suppose for a contradiction that this ``doubly optimal'' coupling satisfies $\chi\left(\Gamma\right)< \pmin$. The set ${A}:=\{u\sim x\colon (u,y)\in \Gamma\}$ is not empty, since it contains the first vertex on a geodesic from $x$ to $y$. Thus, 
$\chi(A\times V) = P(x,A) \ge  \pmin>\chi(\Gamma)$.
This forces  $\chi((A\times V)\setminus\Gamma)>0$, i.e.
\begin{eqnarray}
\label{exists:ab}\exists (x_0,y_0)\in (A\times V)\setminus\Gamma,\quad \chi(x_0,y_0)  & \ge &  \varepsilon,
\end{eqnarray}
for some $\varepsilon>0$. On the other hand, we have
$
\chi(A\times\{y\}) + \chi(A^c\times\{y\})  =  P(y,y) \ \ge \ \frac{1}{2}.
$
This forces $\chi(A^c\times\{y\}) > 0$, because $\chi(A\times\{y\})\le \chi(\Gamma)<\pmin\le \frac{1}{2}$. In other words, 
\begin{eqnarray}
\label{exists:z}
\exists x_1\in A^c,\quad \chi(x_1,y) & \ge & \varepsilon,
\end{eqnarray}
provided $\varepsilon>0$ is chosen small enough. We now use the vertices $x_0,y_0,x_1$ found at (\ref{exists:ab})-(\ref{exists:z}) to construct a new coupling $\widetilde{\chi}$ which contradicts the optimality of $\chi$. For all $(u,v)\in V^2$, we set
\begin{eqnarray*}
\widetilde{\chi}(u,v) & := & \left\{
\begin{array}{ll}
\chi(u,v) & \textrm{ if }u\notin\{x_0,x_1\}\textrm{ and }b\notin\{y_0,y\};\\
\chi(u,v)-\varepsilon & \textrm{ if }(u,v)=(x_0,y_0)\textrm{ or }(u,v)=(x_1,y);\\
\chi(u,v)+\varepsilon & \textrm{ if }(u,v)=(x_0,y)\textrm{ or }(u,v)=(x_1,y_0).
\end{array}
\right.
\end{eqnarray*}
By construction, $\widetilde{\chi}$ is non-negative and has the same marginals as $\chi$. Thus, it is a coupling of $P(x,\cdot)$ and $P(y,\cdot)$. This coupling is moreover optimal, since
\begin{eqnarray*}
\sum_{u,v\in V}\dist(u,v)\left(\widetilde{\chi}(u,v)-\chi(u,v)\right) & = & \varepsilon\left(\dist(x_0,y)+\dist(x_1,y_0)-\dist(x_0,y_0)-\dist(x_1,y)\right)\\
& \le & \varepsilon\left(\dist(x,y)-1+\dist(x_1,y_0)-\dist(x,y)-\dist(x_1,y)\right)\\
& \le & 0,
\end{eqnarray*}
where we have successively used $x_0\in A$, $(x_0,y_0)\notin\Gamma$, and the triangle inequality  $\dist(x_1,y_0)\le \dist(x_1,y)+\dist(y,y_0)$. Finally, since $\Gamma$ contains $(x_1,y)$ but not $(x_0,y_0),(x_1,y)$, we have 
$
\widetilde{\chi}(\Gamma)  \ge  \chi(\Gamma)+\varepsilon,
$
contradicting the double optimality of $\chi$.  
\end{proof}

Our second ingredient is a diffusive hitting-time estimate for super-martingales. Results of this sort are standard, but usually require a uniform bound on the increments, a property which  fails in our directed setting ($\dist(x,x')=\dist(y,y')=1$ no longer implies $\dist(x',y')\le \dist(x,y)+2$).
\begin{lemma}[Hitting-time estimate]\label{lm:super}Let $(Z_t)_{t\ge 0}$ be a discrete-time, $\dN-$valued super-martingale with $Z_0=z_0\in\dN$, and set $\tau:=\min\{t\ge 0\colon Z_t= 0\}$. Assume that almost-surely, 
\begin{eqnarray*}
 \PP\left(Z_{t+1}\ne Z_t | \cF_t\right) & \ge & p {\bf 1}_{(\tau>t)},
\end{eqnarray*}
for all $t\ge 0$, where $(\cF_t)_{t\ge 0}$ denotes the underlying filtration. Then, for all $t\ge 1$, we have
\begin{eqnarray*}
\PP(\tau\ge t) & \le & z_0\sqrt{\frac{10}{pt}}.
\end{eqnarray*}
\end{lemma}
\begin{proof}Our starting point is the following easily verified  inequality: for all $x\in\dR$,
\begin{eqnarray*}
e^{-\frac{x}{2}} & \ge & 1-\frac{x}{2} + \frac{1 \wedge x^2}{10}.
\end{eqnarray*}
In particular, for all $t\in\dN$ and $\lambda\in[0,1]$, we have on the event $\{\tau>t\}$,
\begin{eqnarray*}
\EE\left[e^{-\frac{\lambda}{2}(Z_{t+1}-Z_t)}|\cF_t\right] & \ge  &\EE\left[\left.1-\frac{\lambda}{2}(Z_{t+1}-Z_t)+ \frac{1\wedge  (\lambda Z_{t+1}-\lambda Z_t)^2}{10}\right|\cF_t\right]\\
& \ge & 1+\frac{p\lambda^2}{10}.
\end{eqnarray*}
where the second line uses our assumptions on $Z$. It follows by induction that for all $t\in\dN$,
\begin{eqnarray*}
\EE\left[e^{-\frac{\lambda}{2} (Z_{t\wedge\tau}-Z_0)}\left(1+\frac{p\lambda^2}{10}\right)^{-t\wedge \tau}\right] & \ge & 1.
\end{eqnarray*}
Since $Z_{t\wedge\tau}-Z_0\ge -Z_0=-z_0$, we deduce that
\begin{eqnarray*}
\EE\left[\left(1+\frac{p\lambda^2}{10}\right)^{-t\wedge \tau}\right] & \ge & e^{-\frac{\lambda z_0}{2}} \ \ge \ 1-\frac{\lambda z_0}{2}.
\end{eqnarray*}
In particular, 
\begin{eqnarray*}
{\lambda z_0} & \ge & 2\EE\left[1-\left(1+\frac{p\lambda^2}{10}\right)^{-t\wedge \tau}\right]\\
& \ge & 2\left(1-\left(1+\frac{p\lambda^2}{10}\right)^{-t}\right)\PP(\tau\ge t)\\
& \ge &2\left(1-\left(1+\frac{p t\lambda^2}{10}\right)^{-1}\right)\PP(\tau\ge t).
\end{eqnarray*}
The result now readily follows by choosing $\lambda= \sqrt{10/pt}$. Note that this   choice satisfies $\lambda\in[0,1]$ only when $10/pt\le 1$, but the conclusion trivially holds when $10/pt> 1$.
\end{proof}
\begin{proof}[Proof of Theorem \ref{th:tvdecay}]For each $(x,y)\in V$, let $\chi_{x,y}$ denote an optimal coupling of $P(x,\cdot)$ and $P(y,\cdot)$ satisfying the condition in Lemma \ref{lm:coupling}, and let us define a transition matrix $K$ on $V^2$ by
\begin{eqnarray*}
K((x,y),(u,v)) & :=& \chi_{x,y}(u,v). 
\end{eqnarray*}
Finally, consider a Markov chain $(X_t,Y_t)_{t\ge 0}$ with transition matrix $K$, and let us use the notation  $\EE_{(x,y)}[\cdot]$ to indicate that $(X_0,Y_0)=(x,y)$. By construction, we have for all $x,y\in V$,
\begin{eqnarray*}
\EE_{(x,y)}[\dist(X_{1},Y_{1})] & = & W\left(P(x,\cdot),P(y,\cdot)\right)\  \le \ \dist(x,y);\\
\PP_{(x,y)}[\dist(X_1,Y_1)<\dist(x,y)] & \ge & \pmin{\bf 1}_{(x\ne y)}.
\end{eqnarray*}
By the Markov property, the first condition implies that the process $Z:=(Z_t)_{t\ge 0}$ defined by 
$
Z_t :=  \dist(X_t,Y_t)
$
is a super-martingale w.r.t. the natural filtration $(\cF_t)_{t\ge 0}$ of $(X_t,Y_t)_{t\ge 0}$, and the second  implies that $\PP(Z_{t+1}\ne Z_t|\cF_t)\ge\pmin{\bf 1}_{Z_t\ne 0}$. Thus, Lemma \ref{lm:super} applies and yields 
\begin{eqnarray*}
\PP_{x,y}(X_t\ne Y_t) & \le & \dist(x,y)\sqrt{\frac{10}{ (t+1)\pmin}}.
\end{eqnarray*}
The result follows, since under $\PP_{(x,y)}(\cdot)$, the pair $(X_t,Y_t)$ is a coupling of  $P^t(x,\cdot)$ and $P^t(y,\cdot)$. 
\end{proof}

\bibliographystyle{plain}
\bibliography{draft}

\begin{thebibliography}{10}

\bibitem{MR1262979}
Noga Alon and Yuval Roichman.
\newblock Random {C}ayley graphs and expanders.
\newblock {\em Random Structures Algorithms}, 5(2):271--284, 1994.

\bibitem{bauer2015li}
Frank Bauer, Paul Horn, Yong Lin, Gabor Lippner, Dan Mangoubi, Shing-Tung Yau,
  et~al.
\newblock Li-yau inequality on graphs.
\newblock {\em Journal of Differential Geometry}, 99(3):359--405, 2015.

\bibitem{MR3936154}
Nathana\"{e}l Berestycki and Bat\i \c{S}eng\"{u}l.
\newblock Cutoff for conjugacy-invariant random walks on the permutation group.
\newblock {\em Probab. Theory Related Fields}, 173(3-4):1197--1241, 2019.

\bibitem{MR2316551}
Magnus Bordewich and Martin Dyer.
\newblock Path coupling without contraction.
\newblock {\em J. Discrete Algorithms}, 5(2):280--292, 2007.

\bibitem{MR3439705}
Emmanuel Breuillard and Matthew C.~H. Tointon.
\newblock Nilprogressions and groups with moderate growth.
\newblock {\em Adv. Math.}, 289:1008--1055, 2016.

\bibitem{MR1254308}
P.~Diaconis and L.~Saloff-Coste.
\newblock Moderate growth and random walk on finite groups.
\newblock {\em Geom. Funct. Anal.}, 4(1):1--36, 1994.

\bibitem{MR1374011}
Persi Diaconis.
\newblock The cutoff phenomenon in finite {M}arkov chains.
\newblock {\em Proc. Nat. Acad. Sci. U.S.A.}, 93(4):1659--1664, 1996.

\bibitem{eidi2020ollivierJournal}
Marzieh Eidi and J{\"u}rgen Jost.
\newblock {Ollivier ricci curvature of directed hypergraphs}.
\newblock {\em Scientific Reports}, 10(1):1--14, 2020.

\bibitem{erbar2012ricci}
Matthias Erbar and Jan Maas.
\newblock Ricci curvature of finite {M}arkov chains via convexity of the
  entropy.
\newblock {\em Archive for Rational Mechanics and Analysis}, pages 1--42, 2012.

\bibitem{forman2003bochner}
Robin Forman.
\newblock Bochner's method for cell complexes and combinatorial ricci
  curvature.
\newblock {\em Discrete and Computational Geometry}, 29(3):323--374, 2003.

\bibitem{jost2017a}
J.~Jost.
\newblock {\em {Riemannian geometry and geometric analysis}}.
\newblock Springer, ${}^7$2017.

\bibitem{jost2021characterizations}
J{\"u}rgen Jost and Florentin M{\"u}nch.
\newblock Characterizations of forman curvature.
\newblock {\em arXiv preprint arXiv:2110.04554}, 2021.

\bibitem{MR3127886}
James~R. Lee and Yuval Peres.
\newblock Harmonic maps on amenable groups and a diffusive lower bound for
  random walks.
\newblock {\em Ann. Probab.}, 41(5):3392--3419, 2013.

\bibitem{MR3726904}
David~A. Levin and Yuval Peres.
\newblock {\em Markov chains and mixing times}.
\newblock American Mathematical Society, Providence, RI, 2017.
\newblock Second edition of [ MR2466937], With contributions by Elizabeth L.
  Wilmer, With a chapter on ``Coupling from the past'' by James G. Propp and
  David B. Wilson.

\bibitem{lin2015equivalent}
Yong Lin and Shuang Liu.
\newblock {Equivalent properties of CD inequality on graph}.
\newblock {\em arXiv preprint arXiv:1512.02677}, 2015.

\bibitem{lin2010ricci}
Yong Lin and Shing-Tung Yau.
\newblock {Ricci curvature and eigenvalue estimate on locally finite graphs}.
\newblock {\em Mathematical research letters}, 17(2):343--356, 2010.

\bibitem{MR2341319}
Ravi Montenegro and Prasad Tetali.
\newblock Mathematical aspects of mixing times in {M}arkov chains.
\newblock {\em Found. Trends Theor. Comput. Sci.}, 1(3):x+121, 2006.

\bibitem{munch2019li}
Florentin M{\"u}nch.
\newblock {Li-Yau inequality under $ CD (0, n) $ on graphs}.
\newblock {\em arXiv preprint arXiv:1909.10242}, 2019.

\bibitem{munch2019non}
Florentin M{\"u}nch.
\newblock {Non-negative Ollivier curvature on graphs, reverse Poincar{\'e}
  inequality, Buser inequality, Liouville property, Harnack inequality and
  eigenvalue estimates}.
\newblock {\em arXiv preprint arXiv:1907.13514}, 2019.

\bibitem{munch2017ollivier}
Florentin M{\"u}nch and Rados{\l}aw~K Wojciechowski.
\newblock {Ollivier Ricci curvature for general graph Laplacians: Heat
  equation, Laplacian comparison, non-explosion and diameter bounds}.
\newblock {\em Advances in Mathematics}, 356:106759, 2019.

\bibitem{najman2017modern}
Laurent Najman and Pascal Romon.
\newblock Modern approaches to discrete curvature.
\newblock {\em Lecture Notes in Mathematics, Springer, Berlin}, 2017.

\bibitem{ollivier2010survey}
Yann Ollivier.
\newblock {\em A survey of {R}icci curvature for metric spaces and {M}arkov
  chains}, volume~57 of {\em Adv. Stud. Pure Math.}
\newblock Math. Soc. Japan, Tokyo.

\bibitem{ollivier2009ricci}
Yann Ollivier.
\newblock Ricci curvature of {M}arkov chains on metric spaces.
\newblock {\em Journal of Functional Analysis}, 256(3):810--864, 2009.

\bibitem{ozawa2020geometric}
Ryunosuke Ozawa, Yohei Sakurai, and Taiki Yamada.
\newblock {Geometric and spectral properties of directed graphs under a lower
  Ricci curvature bound}.
\newblock {\em Calculus of Variations and Partial Differential Equations},
  59(4):1--39, 2020.

\bibitem{ozawa2022heat}
Ryunosuke Ozawa, Yohei Sakurai, and Taiki Yamada.
\newblock {Heat flow and concentration of measure on directed graphs with a
  lower Ricci curvature bound}.
\newblock {\em Potential Analysis}, pages 1--15, 2022.

\bibitem{salez2021cutoff}
Justin Salez.
\newblock Cutoff for non-negatively curved {M}arkov chains.
\newblock {\em arXiv preprint arXiv:2102.05597}, 2021.

\bibitem{salez2021sparse}
Justin Salez.
\newblock Sparse expanders have negative curvature.
\newblock {\em arXiv preprint arXiv:2101.08242}, 2021.

\bibitem{schmuckenschlager1998curvature}
Michael Schmuckenschl{\"a}ger.
\newblock Curvature of nonlocal {M}arkov generators.
\newblock {\em Convex geometric analysis (Berkeley, CA, 1996)}, 34:189--197,
  1998.

\bibitem{MR4253426}
Romain Tessera and Matthew C.~H. Tointon.
\newblock A finitary structure theorem for vertex-transitive graphs of
  polynomial growth.
\newblock {\em Combinatorica}, 41(2):263--298, 2021.

\bibitem{villani2009optimal}
C{\'e}dric Villani.
\newblock {\em Optimal transport: old and new}, volume 338.
\newblock Springer, 2009.

\bibitem{yamada2016ricci}
Taiki Yamada.
\newblock {The Ricci curvature on directed graphs}.
\newblock {\em arXiv preprint arXiv:1602.07779}, 2016.

\end{thebibliography}
\end{document}